\documentclass[10pt]{article}

\usepackage{preamble}

\begin{document}

\newcommand\restr[2]{{
  \left.\kern-\nulldelimiterspace 
  #1 
  \vphantom{\big|} 
  \right|_{#2} 
  }}

\makeatletter
\renewcommand{\@seccntformat}[1]{%
  \ifcsname prefix@#1\endcsname
    \csname prefix@#1\endcsname
  \else
    \csname the#1\endcsname\quad
  \fi}
\makeatother

\makeatletter
\newcommand{\colim@}[2]{%
  \vtop{\m@th\ialign{##\cr
    \hfil$#1\operator@font colim$\hfil\cr
    \noalign{\nointerlineskip\kern1.5\ex@}#2\cr
    \noalign{\nointerlineskip\kern-\ex@}\cr}}%
}
\newcommand{\colim}{%
  \mathop{\mathpalette\colim@{\rightarrowfill@\textstyle}}\nmlimits@
}
\makeatother

\newcommand\rightthreearrow{%
        \mathrel{\vcenter{\mathsurround0pt
                \ialign{##\crcr
                        \noalign{\nointerlineskip}$\rightarrow$\crcr
                        \noalign{\nointerlineskip}$\rightarrow$\crcr
                        \noalign{\nointerlineskip}$\rightarrow$\crcr
                }%
        }}%
}

\title{Tate-valued Characteristic Classes II: Applications}         
\author{Shachar Carmeli and Kiran Luecke}        
\date{\today}          
\maketitle

\begin{abstract}
    We present a construction that manufactures $\E_\infty$ orientations of Tate fixed-point objects together with useful formulas for these maps, and then give a number of applications. For example, we produce a formula for the Frobenius homomorphisms of Thom spectra such as $\MU$ as well as certain lifts of Frobenius. We prove a rigidity property of $\MU$ as a \emph{cyclotomic base}. We construct a general obstruction theory for $\E_n$ complex orientations and establish various non-existence results for $p$-typical $\E_n$ orientations for low values of $p$ and $n$. We end with some miscellaneous further applications.
\end{abstract}

\tableofcontents

\section{Introduction}
\subsection{Background and Main Results}
The \emph{complex bordism} spectrum $\MU$ has played a central role in stable homotopy theory ever since the pioneering work of Adams, Novikov, and Quillen (\cite{Adams}, \cite{Novikov}, \cite{Quillen}). Moreover, $\MU$ is an $\E_\infty$ ring spectrum (\cite{Lewis})---a fact which was implicitly used as early as \cite{Quillen}---and recent perspectives have made it increasingly clear that this $\E_\infty$ structure is also of central significance. Recall that a complex orientation of a homotopy ring spectrum $R$ is a homotopy ring map $\MU\to R$. Thus, a fundamental question in understanding the $\E_\infty$ structure of $\MU$ is: if $R$ is an $\E_\infty$ ring, when does it admit a complex orientation by an $\E_\infty$ map? This question is notoriously difficult but has seen much progress: for example, Ando, Hopkins, and Rezk (\cite{ahr}) establish a correspondence between $\E_\infty$ orientations of $R$ (for suitable $R$) and certain sequences in $\pi_*R$. Hopkins and Lawson (\cite{hoplaw}) constructed an obstruction theory for lifting a homotopy complex orientation to an $\E_\infty$ one. Hahn and Yuan (\cite{hahnyuan}) showed that the \emph{periodic} bordism spectrum $\MUP$ admits at least two different $\E_\infty$-structures, and initiated a program to understand various aspects of $\E_\infty$ orientation theory. More recently, Balderrama (\cite{Balderrama}) establishes the existence of $\E_\infty$ orientations of Morava $E$-theories at heights $\leq 2$, and   Burklund, Schlank, and Yuan (\cite{nullstel}), for algebraically closed Morava $E$-theories of arbitrary height. 

The $\E_\infty$ algebra structure of $\MU$ arises from its presentation as the \emph{Thom spectrum} of an $\E_\infty$ map, namely the $j$-homomorphism $\bu\to\pic(\S)$. We will exploit the $\ku$-module structure of $\bu$ to manufacture $\E_\infty$ orientations, at the cost of passing to a Tate fixed-point object,  generalizing our construction from \cite{tatevaluedchar} to complex representations of general compact Lie groups. More precisely, we give a construction (\ref{sharp_construction}) that takes as input a suitable representation $\V $ of a compact Lie group $G$ and an $\E_\infty$ ring map $\omega:\MU\to R$, and produces an $\E_\infty$ ring map
\[\#_\V(\omega):\MU\to R^\tG,\]
where $(-)^\tG$ is an appropriate Tate fixed-points functor. Crucially, our construction comes with a usable formula (\ref{eulertate_formula}) for the underlying homotopy ring map. With it we are able to transport simple linear algebraic facts about the representation $\V $ into concrete formulas controlling the $\E_\infty$ algebra structure of $R^\tG$. 

A central example of an $\E_\infty$ ring map valued in Tate fixed-points (and one which has received much recent attention) is the \emph{Frobenius homomorphism}, defined by Nikolaus and Scholze (\cite[IV.1]{NikolausScholze}). This is a map $\Fr:R\to R^{tC_p}$ which is natural in the $\E_\infty$ ring $R$ and encodes its cyclic power operations (cf. \cite[IV.1.21]{NikolausScholze}). Our first theorem is that, in the complex oriented context, the Frobenius is a special case of our sharp construction.


\begin{introthm}\label{intro_frob_sharp}(\ref{Frob_is_sharp})
    Let $\rho$ be the complex regular representation of $C_p$. Let $\id$ be the identity map of $\MU$. Then the map
    \[\#_\rho(\id):\MU\to\MU^{tC_p}\]
    gotten by \Cref{sharp_construction} is homotopic to the Frobenius as an $\E_\infty$ ring map.
\end{introthm}
Combined with the formulas we get for sharp maps (\Cref{eulertate_formula}) we immediately get the following corollary.

\begin{cor}(\ref{Frobcoordinate})
    Let $\xu\in (\MU^{tC_p})^*\CP^\infty$ be the coordinate corresponding to the unit map $\MU\to\MU^{tC_p}$. Then the coordinate corresponding to the Frobenius is given by
    \[\Fr(\xu)=\xu\prod_{k=1}^{p-1}\frac{\xu+_\Fu [k]_\Fu(\tu)}{[k]_\Fu(\tu)}.\]
\end{cor}

Thus, one can think of \Cref{intro_frob_sharp} as an $\E_\infty$ lift of this formula. 
\begin{rem}
    We make no claim to originality regarding the formula: it is surely known to experts. As mentioned above, it is equivalent to the knowledge of the cyclic power operations in $\MU$ which goes back to \cite{Quillen}. However, to our knowledge this formula has not appeared in the literature, and we give a self contained proof that does not rely on manifold geometry or transversality arguments (as used by Quillen).
\end{rem} 

Next, we would like to demonstrate the usefulness of this with some applications. First, as another corollary of \Cref{intro_frob_sharp} we get certain lifts of the Frobenius map of the $n$-connected cobordism ring $\MU\l n\r$ for each $n$, which at $n=2$ ($\MU\l 2\r=\MU$) recovers an unpublished result of Hahn-Devalapurkar-Raksit-Yuan\footnote{They originally announced the result at Oberwolfach in 2023 but cf. \cite[Example 7.1.4]{sanath} for an outline of their argument.}. Write $\T$ for the Lie group of unit complex numbers.

\begin{cor}(\ref{lift_Frob_p_localization})
For each $n$, the Frobenius $\Fr:\MU\l n\r\to\MU\l n\r^{tC_p}$ factors, as an $\EE_\infty$-ring map, as a composite of a map $\#_{\pi_p}:\MU\l n\r\to(\MU\l n\r_{(p)})^{t\T}$ and the canonical restriction map $(\MU\l n\r_{(p)})^{t\T}\to\MU\l n\r^{tC_p}$. 
\end{cor}

  $\E_\infty$ rings with such lifts of Frobenius are called \emph{cyclotomic bases} (cf. \cite[Definition 3.2.1]{HRW}); such rings $A$ are of great interest because the relative $\mathrm{THH}$ of an $A$-algebra $R$ admits a cyclotomic structure. Furthermore, a map \emph{of cyclotomic bases} is one which commutes with these lifts of Frobenius. Using our formulas, we prove the following rigidity result, indicating that $\MU$ admits rather few automorphisms as a cyclotomic base.
\begin{introthm}(\ref{cyclo_rigidity})
    The only homotopy ring endomorphism of $\MU$ which can be lifted to a map of cyclotomic bases is the identity map.
\end{introthm}

Next, we use our formulas to build an obstruction theory for $\E_n$ complex orientations of $\E_\infty$ rings (cf. \Cref{prop_final_obsser_obs}). When applied to $p$ typical orientations, it is very much in the spirit of Johnson-Noel's calculations in \cite{JohnsonNoel}, but enjoys a few differences (cf. \Cref{compare_JN}). We use this obstruction theory to prove some explicit results, including the following. Recall that the Quillen idempotent is the map $\MU\to\MU_{(p)}$ which represents the universal $p$-typical formal group law. 

\begin{introthm}(\ref{thm_no_p2_E5}, \ref{QI_p3})
    If $R$ is an $\E_\infty$ ring spectrum with nonzero $T(1)$-localization at $p=2$ then $R$ admits no complex orientation which is both 2-typical and representable by an $\E_5$ map. Moreover, the Quillen idempotent at $p=3$ is not representable by an $\E_5$-map.
\end{introthm}

The best previously known bounds on the commutativity of the Quillen idempotent are due to Lawson and Senger, who showed it is not $\E_7$ at $p=2$ and not $\E_9$ at $p=3$, respectively (\cite[Remark 4.4.7]{LawsonBP}, \cite[Theorem 1.3]{Senger}).

Finally, in \Cref{section_further_app} we present some miscellaneous further applications of our sharp construction.

\subsection{Conventions}
All rings and orientations are $\E_\infty$ unless otherwise stated. $\T$ denotes the Lie group of unit complex numbers. For an $\E_\infty$ ring $R$ we make frequent use of the equivalence between the space of $\E_\infty$ ring maps $\MU\to R$ and the space of nullhomotopies of the $j$-homomorphism $\bu\to\pic(R)$ (cf. e.g. \cite{abghr} or \cite{barthelthom}). For a complex oriented ring $R$ with coordinate $x\in R^*\CP^\infty$, the Tate construction $R^{tC_p}$ admits a complex orientation via the unit map $R\to R^{tC_p}$, and by an abuse of notation we also use $x$ to denote the associated coordinate. As the title suggests, this is a sequel to \cite{tatevaluedchar}, but it can be read independently.

\subsection{Acknowledgments}
We would like to thank an anonymous referee for excellent comments and corrections. We would like to thank Jonathan Beardsley, Sanath Devalapurkar, Shay Ben Moshe, Jeremy Hahn, Connor Malin, Eric Peterson, Charles Rezk, John Rognes, Tomer Schlank, Eunice Sukarto, and Neil Strickland for very helpful conversations. We especially thank Jeremy Hahn for suggesting that our methods might shed light on the commutativity of the Quillen idempotent. We thank the entire Spring 2025 Chicago Malort seminar group for a wonderful discussion about \Cref{intro_frob_sharp}. We thank Sanath Devalapurkar for notifying us about forthcoming joint work with Jeremy Hahn, Arpon Raksit, and Allen Yuan.

S.C. is supported by a research grant from the Center for New Scientists at the Weizmann Institute of Science and BSF grant 2024766. He would like to thank the Azrieli Foundation for their support through an Early Career Faculty Fellowship.

K.L. would like to thank the Isaac Newton Institute for Mathematical Sciences, Cambridge, for support and hospitality during the programme Beyond the Telescope Conjecture where work on this paper was undertaken. This work was supported by EPSRC grant no EP/Z000580/1.

\section{A ``sharp'' construction and the Frobenius of $\MU$}
In this section we produce a formula for the complex orientation of $\MU^{tC_p}$ induced by the Frobenius map $\Fr:\MU\to\MU^{tC_p}$ (\Cref{Frobcoordinate}). Our approach is to present a more general construction for producing $\E_\infty$ maps from Thom spectra to Tate fixed-point objects, together with usable formulas for these maps (\Cref{sharp_construction}). It is a generalization of the ``sharp'' construction of \cite[5.11]{tatevaluedchar}, which was itself inspired by (and named after) Ando-French-Ganter's construction in \cite{afg}. 

\subsection{The sharp construction}

Let $G$ be a compact Lie group and let $\V =\C_\mathrm{triv}\oplus\overline{\V}$ be a complex virtual representation of $G$ such that $\overline{\V}$ contains no trivial summands. By this, we mean that $\overline{\V}$ is the difference $A-B$ of two complex representations $A$ and $B$ of $G$, neither of which contain trivial summands. Write $\sF$ for the family of subgroups $H$ of $G$ for which $\overline{\V}$ (that is, either $A$ or $B$) \emph{does} contain a trivial summand when restricted to $H$. For $\cC\in\calg(\cat^\perf)$, write $\cC[\sF]$ for the thick tensor ideal generated by induced objects $X\otimes G/H$, with $X\in\cC$, $H\in\sF$, and $G$ acting by translation on the $G/H$ factor. Write
\[\Nm_\sF:\cC[\sF]\to\cC^{BG} \]
for the inclusion, and $T_\sF$ for the categorical Tate construction (Verdier quotient)
\[\tu _\sF:\cC^{BG}\to\cC^{t_\sF G}:=\cC^{BG}/\Nm_\sF\]
which on unit objects recovers the classical Tate construction $(-)^{t_\Fu G}$ (cf. \cite{GreenleesMay}). 
\begin{defn}
    For an $\E_\infty$ ring $R$, we write
    \[
    \pic(R)^{t_\sF G}:= \pic(\Perf(R)^{t_\sF G})
    \]
      
\end{defn}

Note that this is not the same as $\pic(R^{t_\sF G})$---indeed, we introduce these categories because, unlike $\Perf(R^{t_\sF G})$, the $\infty$-category $\Perf(R)^{t_\sF G}$ receives the Tate fixed-point construction as a \emph{strong symmetric monoidal} functor. However, to connect to complex orientation theory, we need to define ``$j$-homomorphisms'' for them. Let $\pi:BG\to\pt$ be the projection, and $\pi^*:\Perf(R)\to\Perf(R)^{BG}$ the induced ``trivial action'' map.

\begin{defn}\label{defn_categorical_tate_j}
    Define the \emph{$j$-homomorphism of} $\pic(R)^{t_\sF G}$ to be the composite
    \[\bu\oto{j_R}\pic(R)\oto{\pi^*}\pic(R)^{BG}\oto{T_\sF}\pic(R)^{t_\sF G}.\]
\end{defn}
On objects, this sends $V\mapsto\S^V\otimes R^{t_\sF G}$, where $R^{t_\sF G}$ refers to the unit of $\Perf(R)^{t_\sF G}$.
Let $R$ be an $\E_\infty$ ring and consider the diagram
\begin{equation}\label{display_sharp}
\begin{tikzcd}
 \bu\arrow[r,"\V\otimes(-)"]&\bu^{BG}\arrow[rr, "j_R^{BG}"]& &\mathrm{pic}(R)^{BG}\arrow[r,"T_\sF"]& \mathrm{pic}(R)^{t_\sF G} \\
\end{tikzcd}.
\end{equation}

\begin{lem}\label{composite_is_j}
    The composite in Display (\ref{display_sharp}), which sends $V\mapsto(\S^{\V\otimes V}\otimes R)^{t_\sF G}$, is canonically homotopic to the $j$-homomorphism of $\pic(R)^{t_\sF G}$ (cf. \Cref{defn_categorical_tate_j}).
 
\end{lem}
\begin{proof}
    Recall that for a vector space $V$ there is a cofiber sequence
    $S(V)_+\to\S^0\to S^V$
    where $S(V)$ denotes the unit sphere in a chosen metric on $V$. Because of the assumptions on $\overline{\V}=A-B$ and $\sF$, the unit spheres $S(A)_+, S(B)_+\in\Sp^{BG}$ are built out of $G$-cells of the form $\S^k\times G/H$ with $H\in \sF$, and are thus killed by $(-)^{t_\sF G}$. The same holds for the unit spheres $S(A\otimes V)_+$ and $S(B\otimes V)_+$. Hence, by monoidality of $\mathbb{S}^{(-)}$, the map induced by the inclusion $\C_\triv\to\C_\triv\oplus A$ and the dual of the dual of the map induced by $0\to B$
    \[\S^{\C_\mathrm{triv} \otimes V}\to\S^{(\C_\triv\oplus A)\otimes V}\from\S^{\V\otimes V}
\]
both become isomorphisms after applying $(-)^{t_\sF G}$, since their fibers are contained the thick tensor ideal defining the kernel of that Verdier quotient. Thus, the induced natural transformation
\[
S^V\otimes R^{t_\sF G}=(\S^{\C_\mathrm{triv}\otimes V}\otimes R)^{t_\sF G}\to(\S^{\V\otimes V}\otimes R)^{t_\sF G} 
\] 
is an isomorphism.
\end{proof}

\begin{construction}\label{sharp_construction}
    Let $\omega:\MU\to R$ be an $\E_\infty$ ring map, $G$ a compact Lie group and $\V $ and $\sF$ as above. We will construct an $\E_\infty$ ring map
    \[\sharp_\V(\omega):\MU\to R^{t_\sF G}.\]
    First we claim that it suffices to produce a nullhomotopy of Display (\ref{display_sharp}). Indeed, note that the unit object of $\pic(R)^{t_\sF G}$ is $R^{t_\sF G}$. Thus by  \Cref{composite_is_j}, the space of nullhomotopies of Display (\ref{display_sharp}) is canonically isomorphic to the space of nullhomotopies of the $j$-homomorphism of $\pic(R)^{t_\sF G}$. By connectedness of $\bu$ and the fact that the basepoint of $\pic(R)^{t_\sF G}$ is $R^{t_\sF G}$, this is isomorphic to the space of nullhomotopies of the classical $j$-homomorphism $j_{R^{t_\sF G}}:\bu\to\pic(R^{t_\sF G})$, i.e. $\E_\infty$ ring maps $\MU\to R^{t_\sF G}$.
    
    Finally, to produce our desired nullhomotopy: with a slight abuse of notation, let $\omega$ also denote the nullhomotopy of $j_R:\bu\to\pic(R)$ corresponding to the orientation $\omega$ of $R$. Then the map $\#_\V(\omega)$ corresponds to the nullhomotopy of Display (\ref{display_sharp}) given by 
    \[\tu _\sF\circ \omega^{BG}\circ(\V\otimes(-)).\]
\end{construction}

\begin{rem}\label{rem_old_sharp}
For comparison, in \cite[Section 5.1]{tatevaluedchar} we focused on the special case of $G=\T$, $\V =1-L$ for the standard representation $L$ of $\T$, and $\sF=\{1\}$. This has the added benefit that one can replace the domain of the first map of Display (\ref{display_sharp}) with $\ku$ (because $1-L$ has virtual dimension 0), thereby producing $\MUP$-orientations out of $\MU$-orientations.
\end{rem}

\begin{rem}\label{remark_generalized_sharp}
    There are some obvious generalizations and variations. First, one can replace $\bu$ by any $\ku$-module over $\bu$ (for example, a highly connected cover $\bu\l n\r$) and $\MU$ by the corresponding Thom spectrum (for example, $\MU\l n\r$). In an orthogonal direction, one can also replace $\bu$ with $\bo$, $\MU$ with $\MO$, and the complex representation $\V $ with a real representation. Combining these two directions of generalization with $G=C_2$, $\sF=\{1\}$ and $\V =1-\sigma$, Display (\ref{display_sharp}) becomes
    \[\begin{tikzcd}
\bo\l n\r\arrow[rr,"(1-\sigma)\otimes(-)"]& &\bo\l n+1\r^{BC_2}\arrow[rr, "j_R^{BC_2}"]& &\mathrm{pic}(R)^{BC_2}\arrow[r,"T"]& \mathrm{pic}(R)^{tC_2} \\
\end{tikzcd},\]
where we have used the fact that $1-\sigma$ is in Whitehead filtration 1 so that multiplication by it jumps one stage in the Whitehead filtration. Thus, via the recipe of \Cref{sharp_construction}, this display produces an $\MO\l n\r$-orientation of $R^{tC_2}$ ($\#_\V(\omega):\MO\l n\r\to R^{tC_2}$) for every $\MO\l n+1\r$-orientation of $R$ ($\omega:\MO\l n+1 \r\to R$).
\end{rem}

\subsection{The Euler-Tate map}
In this subsection we derive a formula for the ``Hirzebruch series'' of a map gotten via \Cref{sharp_construction}.

Let $R$ be a complex-orientable commutative ring spectrum. Then $\calg(\MU,R)$ is a torsor for $\Map(\bu,\glone R)$. Thus for every pair of $f,g\in\calg(\MU,R)$ there is a unique element $\frac{f}{g}\in \Map(\bu,\glone R)$ such that 
\[f=g\cdot \frac{f}{g}\in\calg(\MU,R).\]

\begin{defn}\label{defn_ET_map_and_series}
    Let $\omega:\MU\to R, G, \V=\C_\triv\oplus\overline{\V}$, and $\sF$ be such that \Cref{sharp_construction} applies. Let $u:R\to R^{t_\sF G}$ be the unit map, and $u_*\omega:\MU\to R^{t_\sF G}$ the composite orientation. The \emph{Euler-Tate map} of $\#_\V(\omega)$ is defined to be
    \[\tch_{\V ,\omega}:=\frac{\#_\V(\omega)}{u_* \omega}\in\Map(\bu,\glone R^{t_\sF G}).\]
\end{defn}

Thus the Euler-Tate map is what was classically referred to as a ``stable exponential characteristic class'' with values in $R^{t_\sF G}$. A priori it is only defined on rank 0 complex vector bundles. We extend to vector bundles $V\to X$ of positive rank $n$ by the formula
\[\tch_{\V ,\omega}(V):=\tch_{\V ,\omega}(V-n).\]
For a complex oriented ring $\omega:\MU\to R$ write $e_\omega(V)\in R^{2n}(X)$ for the euler class of $V$ in the orientation $\omega$. The next proposition gives a formula for the Euler-Tate map of $\#_\V$ as a ratio of euler classes.

\begin{prop}\label{eulertate_formula}
    Let $V\to X$ be a rank $n$ vector bundle over a finite complex $X$. Let $\overline{\V}$ also denote the induced vector bundle over $BG$. Then
    \[\tch_{\V ,\omega}(V)=\frac{e_{u_*\omega}(V\otimes \overline{\V})}{e_{u_*\omega}(\overline{\V})^n}\in\pi_0\Map(X,\glone R^{t_\sF G}).\]
\end{prop}
\begin{proof}
    The characteristic class $\tch_{\V ,\omega}$ is defined as the ratio of two complex orientations of $R^\tG$. Thus $\tch_{\V ,\omega}(V)$ is the ratio of the two Thom classes, or more precisely, the composite of one Thom isomorphism with the inverse of the other. Since $X$ is finite we have $(R^\tG)^X=(R^X)^\tG$. Write $e_\omega(\sF)$ for the collection of euler classes of $G$-representations $U$ such that $U^K=\{0\}$ for all subgroups $K$ not in $\sF$. By the Greenlees-May inversion formula (\cite[3.20]{GreenleesMaybook})
    the Thom isomorphism $u_*\omega$ in the denominator of \Cref{defn_ET_map_and_series} is gotten from the $\omega$-Thom isomorphism for $V\to X\times BG$ after inverting $e_\omega(\sF)$:
    \[R^*(X\times BG)[e_\omega(\sF)^{-1}]\simeq R^{*+2n}(\S^V\times BG)[e_\omega(\sF)^{-1}]. \]
    Similarly, the numerator isomorphism $\#_\V$ is given by composing the $e_\omega(\sF)$-inverted  $\omega$-Thom isomorphism
    \[
        R^*(X\times BG)[e_\omega(\sF)^{-1}]\simeq R^{*+2n\mathrm{dim}_\C\V}(\S^{V\otimes \V})[e_\omega(\sF)^{-1}] 
    \]
    with the maps induced by the inclusion $i:\S^{\C_\triv}\to\S^{\C_\triv\oplus A}$ and the dual of $j:0\to B$ (recall that $\V=\C_\triv\oplus A\ominus B$)
    \[R^{*+2n\mathrm{dim}_\C\V}(\S^{V\otimes \V})[e_\omega(\sF)^{-1}]\oto{(j^{\vee})^*}R^{*+2n\mathrm{dim}_\C\V}(\S^{V\otimes(\C_\triv\oplus A)}\times BG)[e_\omega(\sF)^{-1}]\]
    \[R^{*+2n\mathrm{dim}_\C\V}(\S^{V\otimes(\C_\triv\oplus A)}\times BG)[e_\omega(\sF)^{-1}]\oto{i^*}R^{*+2n\mathrm{dim}_\C\V}(\S^{V}\times BG)[e_\omega(\sF)^{-1}]\]
    Thus the relevant composite of these two Thom isomorphisms is precisely the $u_*\omega$-euler class of $V\otimes\overline{\V}$. The denominator of $e_{u_*\omega}(\overline{\V})^n$ is acquired since one must first shift $V$ to be rank zero by subtracting a rank $n$ trivial bundle.
\end{proof}

\subsection{The complex oriented Frobenius}\label{section_complex_Frob_formulas}

In this section we illustrate how \Cref{sharp_construction} and \Cref{eulertate_formula} can be used to obtain a formula for the Frobenius homomorphism of $\MU$. 

Recall that for any $\E_\infty$ ring $R$ and prime $p$ Nikolaus and Scholze (cf. \cite[IV.1]{NikolausScholze}) define a \emph{Frobenius homomorphism}
\[\Fr:R\to R^{tC_p}\]
where $(-)^{tC_p}$ is taken with respect to the trivial action. The Frobenius is natural in $R$, and thus encodes quite a lot of the $\E_\infty$ ring structure of $R$---in particular, it encodes much of the data of the power operations of $R$ (cf. \cite[IV.1.21]{NikolausScholze}). It is thus very favorable to have formulas for the Frobenius homomorphism. As it turns out, the Frobenius is a special case of the $\#$ construction.

Let $\rho$ be the regular representation of $C_p$. We may apply \Cref{sharp_construction} with $G=C_p$, $\V =\rho$, and $\sF=\{1\}$ (so that $(-)^{t_\sF C_p}=(-)^{tC_p}$) to obtain an $\E_\infty$ ring map
\[\#_\rho:\MU\to\MU^{tC_p}.\]
\begin{thm}\label{Frob_is_sharp}
    $\#_\rho$ and $\Fr$ are homotopic as $\E_\infty$ ring maps.
\end{thm}
\begin{proof}
    Consider the following diagram.
\[\begin{tikzcd}
    \bu\arrow[r,"\otimes\rho=(-)^{\oplus p}"]\arrow[d]&\bu^{BC_p}\arrow[d]& 
    \\
    \mathrm{pic}(\S)\arrow[r, "(-)^{\otimes p}"]\arrow[d]&\mathrm{pic}(\S)^{BC_p}\arrow[d]&
    \\
    \mathrm{pic}(\MU)\arrow[r, "(-)_\MU^{\otimes p}"]& \mathrm{pic}(\MU)^{BC_p}\arrow[r, "T"]&\mathrm{pic}(\MU)^{tC_p} \\
\end{tikzcd}\]
First note that both columns are canonically null, via nullhomotopies we will call $\omega_\id$ (left column) and $\omega_\id^{BC_p}$ (right column). By functoriality the rectangle commutes up to canonical homotopy. That is, the two outer paths of the rectangle $\bu\to\pic(\MU)^{BC_p}$ are canonically homotopic \emph{and} the nullhomotopies of each path (given by the relevant pre/post compositions of $\omega_\id$ and $\omega_\id^{BC_p}$) are homotopic. 

By \cite[Corollary 2.14]{CarmeliPic} on unit objects the bottom horizontal composite is homotopic to the Frobenius $\Fr:\MU\to\MU^{tC_p}$ . Thus, the Frobenius corresponds via the universal property of $\MU$  to the nullhomotopy of the ``L''-shaped path from top left to bottom right given by the canonical nullhomotopy of the left column. By what we have concluded about the rectangle, this is homotopic to the nullhomotopy of ``Z''-shaped composite from top left to bottom right given by the canonical nullhomotopy of the right column, which is the orientation corresponding to $\#_\rho$ by definition.
\end{proof}

\begin{rem}
   There is an alternate proof, due to the Chicago Malort seminar, using the presentation of $\MU$ as the pushout $\S\leftarrow\S[U]\oto{j}\S$ in $\calg$ defining the $\EE_\infty$ quotient $\S\modmod U$. By naturality of the Frobenius this reduces the question to the trivial Thom spectra $\S[U]$ and $\S$, whose Frobenius maps are easier to transport into a geometric formulation (because suspension spectra have diagonals as opposed to mere \emph{Tate} diagonals). 
\end{rem}


As an immediate consequence we have the following. Let $\xu\in\MU^*\CP^\infty$ be the canonical coordinate. Write $\xu+_\Fu \yu\in \MU_*[\![\xu ,\yu]\!]$ for the universal formal group law. Recall that for $n\geq 0$ the universal $n$-\emph{series} 
$
[n]_\Fu (\xu)\in\MU_*[\![\xu]\!]
$
is defined inductively by setting 
\[
[0]_\Fu (\xu)=0 \quad \text{and} \quad [n+1]_\Fu (\xu )=[n]_\Fu (\xu )+_\Fu \xu.
\] 
By an abuse of notation, we also use $\xu$ to denote the coordinate induced by the unit map of the Tate construction $\MU\oto{\mathrm{unit}}\MU^{tC_p}$. Finally, we let $\tu\in \MU^2BC_p$ denote the pullback of the canonical coordinate along $BC_p \to \CP^\infty$ and use the same notation for its image in $\pi_{-2}\MU^{tC_p}$.  
These classes give an identification 
\[   (\MU^{tC_p})^*\CP^\infty\simeq \frac{\MU_*(\!(\tu )\!)}{([p]_\Fu (\tu ))}[\![\xu ]\!].
\]
\begin{cor}\label{Frobcoordinate}
    
    The complex orientation of $\MU^{tC_p}$ induced by the Frobenius is given by the coordinate
    \[\Fr(\xu)=\xu\prod_{k=1}^{p-1}\frac{\xu+_\Fu [k]_\Fu (\tu )}{[k]_\Fu (\tu )}.\]
\end{cor}
\begin{proof}
    By \Cref{Frob_is_sharp}, $\Fr(\xu)=\#_\rho(\xu)$. By the definition of the Euler-Tate map, we have
    \[\tch_\rho(L-1)=\frac{\#_\rho(\xu)}{\xu}.\]
    Let $\xi\colon C_p \to \CC^\times$ be the standard representation given by $\xi(k) = e^{\frac{2 \pi i k}{p}} $, so that $\rho=\oplus_{k=0}^{p-1}\xi^k$. Note that the Euler class of $\xi$ is $e_u(\xi)=\tu$. Using \Cref{eulertate_formula} and the multiplicativity of the Euler class, we find that 
    \[
    \tch_\rho(L-1)=\prod_{k=1}^{p-1}\frac{e_u(L\otimes\xi^k)}{e_u(\xi^k)}=\prod_{k=1}^{p-1}\frac{\xu+_\Fu [k]_\Fu (\tu )}{[k]_\Fu (\tu )},
    \]
    so that 
    \[
    \Fr(\xu) = \#_\rho(\xu) = \xu\cdot \tch_\rho(L-1) = \xu\cdot \prod_{k=1}^{p-1}\frac{\xu+_\Fu [k]_\Fu (\tu )}{[k]_\Fu (\tu )}.
    \]
\end{proof}

\begin{rem}
    The $\F_p^\times$-invariance of the Frobenius (cf. \cite[IV.1.4]{NikolausScholze}) is manifest in the above formula---$\F_p^\times$ acts by sending $\tu$ to $[j]_\Fu (\tu )$ for $j\in \{1,..,p-1\}$, which simply permutes the factors in the product.
\end{rem}

\begin{rem}\label{evensteenrodalg}
    One can recover a good deal of information about the Frobenius of $\F_p$ as a map of cohomology theories $\F_p\to\F_p^{tC_p}$. Write $\xa $ and $\ta $ for images of $\xu$ and $\tu$ under the canonical map $\MU\to\F_p$. The formula of \Cref{Frobcoordinate} pushes forward under the $\E_\infty$ orientation $\MU\to\F_p$ to the very manageable 
    \[
    \Fr(\xa )=\xa \prod_{k=1}^{p-1}\frac{\xa +k\ta }{k\ta } = \frac{\xa \ta ^{p-1}-\xa ^p}{\ta ^{p-1}}.
    \]
    It is only partial information because $\MU\to \F_p$ is not surjective as a map of cohomology theories. However, it sees the entire evenly generated part of the mod $p$ Steenrod algebra: if one defines the $P^i$ as the coefficients of $(-1)^i\ta ^{-i(p-1)}$ in the evident decomposition of $\F_p^{tC_p}$ (cf. \cite[IV.1.15]{NikolausScholze}), then the square that witnesses the Frobenius self-commuting\footnote{More precisely, we mean the square that relates the Frobenius of $\F_p$, the induced map on $(-)^{tC_p}$, and the Frobenius of $\F_p^{tC_p}$.} produces the Adem relations.
\end{rem}

Because the Frobenius is natural along $\E_\infty$ ring maps, the difference $\Fr^A\circ\phi-\phi^{tC_p}\circ\Fr^\MU$ is a first obstruction to a homotopy ring map $\phi:\MU\to A$ admitting an $\E_\infty$ structure. We will refine this idea in \Cref{obstruction_section}. For now, we will record a formula for this difference using \Cref{Frobcoordinate}.

\begin{notation}\label{pushforward_series}
For a power series $g(z) = g_0 + g_1z  + \dots$ over a ring $R$ and a ring map $\psi \colon R\to R'$ we denote by $\psi(g)(z)$ the power series $\psi(g_0) + \psi(g_1)z + \dots$, i.e., the power series obtained by applying $\psi$ to all the coefficients of $g(z)$. 
\end{notation}

\begin{prop}\label{frob_commute_formula}   
 Let $u:\MU\to A$ be an $\E_\infty$ ring map with coordinate $x$. Let $\phi:\MU\to A$ be a homotopy ring map and $f(\xa)\in A_*[\![x]\!]$ defined by $\phi(\xu)=f(x)$ be as above. Let $p$ be any prime. Then in $A_*(\!(\ta)\!)/([p]_{\Fa }(\ta))[\![\xa]\!]$, the element $\Fr^A(\phi(\xu))-\phi^{tC_p}(\Fr^\MU(\xu))$ is given by 
\[
\Fr^A(f)\left( \xa\cdot \prod_{k=1}^{p-1}\frac{ \xa+_{\Fa }[k]_{\Fa }(\ta)}{[k]_{\Fa }(\ta)} \right)-f(\xa)\cdot\prod_{k=1}^{p-1}\frac{f(\xa+_{\Fa }[k]_{\Fa }(\ta))}{f([k]_{\Fa }(\ta))}.
\]
\end{prop}
\begin{proof}

We will expand $\Fr^A(\phi(\xu))$ and $\phi^{tC_p}(\Fr^\MU(\xu))$ using \Cref{Frobcoordinate}. We will drop the superscript on $\phi^{tC_p}$ to streamline notation.

\begin{itemize}
\item For the former, we have 
\[
\Fr^A(\phi(\xu)) = \Fr^A(f(\xa)) = \Fr^A(f)(\Fr^A(\xa)) = \Fr^A(f)(u(\Fr^\MU(\xu))) = \Fr^A(f)\left(\xa\cdot \prod_{k=1}^{p-1}\frac{\xa+_{\Fa }[k]_{\Fa }(\ta)}{[k]_{\Fa }(\ta)}\right).
\]
\item For the latter, we have
\[
\phi(\Fr^\MU(\xu)) = \phi(\xu\cdot \prod_{k=1}^{p-1} \frac{\xu+_\Fu  [k]_\Fu(\tu)}{[k]_\Fu(\tu)}) = \phi(\xu)\cdot \prod_{k=1}^{p-1} \frac{\phi(\xu+_\Fu  [k]_\Fu (\tu ))}{\phi([k]_\Fu (\tu ))}. 
\]
It remains to show that $\phi(\xu+_\Fu  [k]_\Fu (\tu )) = f(\xa+_{\Fa } [k]_{\Fa }(\ta))$ and that $\phi([k]_{F}(\tu )) = f([k]_{\Fa }(\ta))$, which follows from \Cref{eq:phi_on_coordinates_sum}.
\end{itemize}
\end{proof}

\begin{lem}\label{eq:phi_on_coordinates_sum}
Let $\xu,\yu$ (resp. $\xa,\ya$) be the two pullbacks of the coordinate $\xu\in \MU^*\CP^\infty$ (resp. $\xa\in A^*\CP^\infty$) along the two projection maps $\CP^\infty\times \CP^\infty\to\CP^\infty$, so that 
\[
\MU^*(\CP^\infty \times \CP^\infty) \cong \MU_*[\![\xu ,\yu ]\!],\ \ \text{(resp. }\  
A^*(\CP^\infty \times \CP^\infty) \cong A_*[\![\xa,\ya]\!]). 
\]
Then, for all $k,\ell \in \ZZ$ we have 
\[
\phi\Big([k]_\Fu (\xu) +_\Fu  [\ell]_\Fu (\yu)\Big) = f\Big([k]_{\Fa }(\xa)+_{\Fa } [\ell]_{\Fa }(\ya)\Big) 
\]
in $A^*(\CP^\infty\times \CP^\infty)$. 
\end{lem}

\begin{proof}
Consider the map $\rho_{k,\ell}\colon \CP^\infty\times \CP^\infty \to \CP^\infty$ given on complex lines  by $(L,L')\mapsto L^{\otimes k}\otimes (L')^{\otimes \ell}$. 
Then, the induced maps 
\[
\rho_{k,\ell,\MU}^*\colon \MU^*(\CP^\infty) \to \MU^*(\CP^\infty\times \CP^\infty)
\]
\[
\rho_{k,\ell,A}^*\colon A^*(\CP^\infty) \to A^*(\CP^\infty\times \CP^\infty)
\]
satisfy 
\[
\rho_{k,\ell,\MU}^*(\xu) = [k]_\Fu (\xu) +_\Fu  [\ell]_\Fu (\yu).
\]
\[
\rho_{k,\ell,A}^*(\xa) = [k]_{\Fa }(\xa) +_{\Fa } [\ell]_{\Fa }(\ya).
\]
Since $\phi$ is a morphism of spectra, it satisfies $\phi\circ \rho_{k,\ell,\MU}^* = \rho_{k,\ell,A}^*\circ \phi$. Hence we get
\[
\phi([k]_\Fu (\xu) +_\Fu  [\ell]_\Fu (\yu)) = \phi\circ  \rho_{k,\ell,\MU}^*(\xu) = \rho_{k,\ell,A}^*f(\xa) = f(\rho_{k,\ell,A}^*(\xa)) = f([k]_{\Fa }(\xa) +_{\Fa } [\ell]_\Fu (\ya)).
\]
\end{proof}

\section{Lifts of Frobenius}

In this section we produce some ``integral lifts'' of Frobenius. $\E_\infty$ rings with such lifts are called \emph{cyclotomic} $\E_\infty$ rings by Hahn-Devalapurkar-Raksit-Yuan in forthcoming work. We show that $\MU$ is a cyclotomic $\E_\infty$ ring (a result which is known to the aforementioned four people). Moreover, we show that as such, $\MU$ admits rather few endomorphisms.

The basic observation is that the regular representation $\rho$ of $C_p$ is the restriction (along the inclusion $C_p\to \T$) of a representation of $\T$ which fits the criteria for \Cref{sharp_construction}. This puts us in the situation of the following lemma.
\begin{lem}\label{lem_compare_sharps}
    Fix $n>0$. Let $i:G_1\to G_2$ be a map of compact Lie groups. Let $G_1, \V_1, \sF_1$ and $G_2, \V_2, \sF_2$ be data satisfying the conditions of \Cref{sharp_construction}, and such that $\sF_1=i^*\sF_2$, $\V _1=i^*\V_2$. Then for every $\omega:\MU\l n\r\to R$ there is a commutative diagram in $\calg$
    \[\begin{tikzcd}
    & &R^{t_{\sF_2} G_2}\arrow[d, "i^*"] \\
   \MU\l n\r\arrow[urr,"\#_{\V_2}(\omega)"]\arrow[rr,"\#_{\V_1}(\omega)"'] & & R^{t_{\sF_1} G_1}\\
\end{tikzcd}\]
\end{lem}
\begin{proof}
    Because of the assumptions on the data $G_1, \V_1, \sF_1$ and $G_2, \V_2, \sF_2$,  the homomorphism $i$ induces a map between instances of Display (\ref{display_sharp}) (cf. also \Cref{remark_generalized_sharp}), presented by the following commutative diagram
    \[\begin{tikzcd}
\bu\l n\r \arrow[r,"\V_2\otimes(-)"]\arrow[d]&\bu\l n\r^{BG_2}\arrow[rr, "j_R^{BG_2}"]\arrow[d, "i^*"]& &\mathrm{pic}(R)^{BG_2}\arrow[r,"T_{\sF_2}"]\arrow[d, "i^*"]& \mathrm{pic}(R)^{t_{\sF_2} G_2}\arrow[d, "i^*"] \\
\bu\l n\r\arrow[r,"\V_1\otimes(-)"]&\bu\l n\r^{BG_1}\arrow[rr, "j_R^{BG_1}"]& &\mathrm{pic}(R)^{BG_1}\arrow[r,"T_{\sF_1}"]& \mathrm{pic}(R)^{t_{\sF_1} G_1} \\
\end{tikzcd}.\]
Thus the nullhomotopies defining $\#_{\V_1}(\omega)$ and $i^*\circ\#_{\V_2}(\omega)$ are homotopic via the homotopy witnessing the commutativity of the diagram.
\end{proof}

As a special case, we get our desired lifts of Frobenius. For each prime $p$, let $\pi_p$ be the representation of $\T$ given by 
\[\pi_p=\C\oplus L\oplus L^2\oplus...\oplus L^{p-1}.\]

\begin{thm}\label{thm_lift_of_Frob}
    Let $i:C_p\to \T$ be the inclusion. Let $\sF_p$ be the family of subgroups of $\T$ given by $\{1,C_2,...,C_{p-1}\}$. There is a commutative diagram in $\calg$
    \[\begin{tikzcd}
    & &\MU\l n\r^{t_{\sF_p}\T}\arrow[d, "i^*"] \\
   \MU\l n\r\arrow[urr,"\#_{\pi_p}"]\arrow[rr,"\Fr"'] & & \MU\l n\r^{tC_p}\\
\end{tikzcd}.\]
\end{thm}
\begin{proof}
    Apply \Cref{lem_compare_sharps} with $(G_1,\V_1,\sF_1)=(C_p, \rho, \{1\})$,  $(G_2,\V_2,\sF_2)=(\T, \pi_p, \sF_p)$, and $i:C_p\to \T$ the inclusion. Then use \Cref{Frob_is_sharp} (whose proof remains unchanged for $\MU\l n\r$) to replace $\#_\rho$ with $\Fr$.
\end{proof}

We immediately obtain the following corollary.

\begin{cor}\label{lift_Frob_p_localization}
    For each $n$, the Frobenius map $\MU\l n\r\to\MU\l n\r^{tC_p}$ factors through the canonical diagram 
    \[(\MU\l n\r_{(p)})^{t\T}\to(\MU\l n\r_{p})^{t\T}\to\MU\l n\r^{tC_p}\]
    via the map $\#_{\pi_p}$ of \Cref{thm_lift_of_Frob}.
\end{cor}
\begin{proof}
The claim is simply that for any $p$-local ring $R$, $R^{t\T}\simeq R^{t_{\sF_p}\T}$. Indeed, the family $\sF_p$ consists of groups $H$ of order prime to $p$. The composite of transfer and restriction
    \[R[\T/H]\oto{\mathrm{tr}}R[\T]\oto{\mathrm{res}} R[\T/H],\]
    is given by multiplication by the order $|H|$, which is invertible. Thus the outer terms are a retract of the middle, and so the thick tensor ideals defining $(-)^{t\T}$ and $(-)^{t_{\sF_p}\T}$ coincide.
\end{proof}
\begin{rem}
    Note that at $p=2$, the family $\sF_2=\{1\}$ is trivial, so no localization is needed---the Frobenius factors through $\MU\l n\r^{t\T}$. Moreover, when $p>2$ and $n=2$ (so that $\MU\l n\r=\MU$), one can set $k=(p-1)/2$ and perform the above construction with the representation $\bigoplus_{-k}^kL^j$ (whose restriction to $C_p$ is still the regular representation) to obtain a \emph{different} lift of Frobenius, which even for $p=3$ does not require a localization or completion (since $-1$ and 1 are already invertible). We thank the referee for pointing this out.
\end{rem}
\begin{rem}
    Recall that for any $\E_\infty$ ring $R$ the Frobenius $\Fr_R:R\to R^{tC_p}$ factors (naturally and canonically) through the $\F_p^\times$ fixed points $(R^{tC_p})^{h\F_p^\times}\to R^{tC_p}$ (\cite[IV.1.4]{NikolausScholze}). It is thus natural to ask if a similar result holds for these lifts of Frobenius. 
    Although $\MU^{t\T}_p$ does admit a $\F_p^\times$-action via roots of unity $\F_p^\times\subset \Z_p^\times$ which is compatible with the map $\MU^{t\T}_p\to \MU^{tC_p}$, we cannot hope to factor our lift of Frobenius $\#_{\pi_p}$ through $(\MU^{t\T}_p)^{h\F_p^\times}$. Indeed, a root of unity $\xi$ sends $[k]_\Fu(\tu)$ to $[\xi\cdot k]_\Fu(\tu)$ and the formula of \Cref{Frobcoordinate} is clearly not invariant under this unless one is working modulo $[p]_\Fu(\tu)$. 
\end{rem}

We do get a weak form of $\F_p^\times$-invariance, which we record for later use. 

\begin{defn}\label{defn_Ip}
    For each prime $p$ define the $\E_\infty$ ring $I_p$ via the pullback diagram
\[\begin{tikzcd}
    I_p\arrow[r]\arrow[d]&(\MU^{tC_p})^{h\F_p^\times}\arrow[d] \\
    \MU^{t\T}_p\arrow[r]& \MU^{tC_p}.\\
\end{tikzcd}\]
\end{defn}
Note that $\pi_*I_p$ is isomorphic to the subgroup of $\pi_*MU^{t\T}_p$ consisting of elements which are $\F_p^\times$-invariant modulo $[p]_\Fu(\tu)$.

\begin{prop}\label{lift_Frob_Fptimes_inv}
    The lifts of Frobenius $\#_{\pi_p}:\MU\to \MU^{t\T}_p$ factor canonically through
    $I_p\to \MU^{t\T}_p$.
\end{prop}
\begin{proof}
    This follows immediately from the definition of $I_p$ as a pullback, the lifts of Frobenius of \Cref{thm_lift_of_Frob}, and the $\F_p^\times$-invariance of the Frobenius of \cite[IV.1.4]{NikolausScholze}.
\end{proof}

Following Hahn-Devalapurkar-Raksit-Yuan and others, we make the following definition.

\begin{defn}(cf. \cite[Definition 3.2.1]{HRW})
    A bounded below $\E_\infty$ ring $R$ is called a \emph{cyclotomic base} if it is equipped with factorizations $\psi_p$ as in the diagram
    \[\begin{tikzcd}
         & (R^{t\T})_p\arrow[d]\\
        R\arrow[ru,dashed,"\psi_p"]\arrow[r,"\Fr"] &R^{tC_p}. \\
    \end{tikzcd}\]
    A \emph{cyclotomic base map} $\phi:R\to R'$ is one which commutes with the lifts of Frobenius---that is $\psi_p\circ\phi \simeq \phi^{t\T}\circ \psi_p$.
\end{defn}

\begin{rem}
    The indicated definition of Hahn-Raksit-Wilson may appear different at first glance. They define a cyclotomic base structure on $R$ as a diagram of $\T$-equivariant $\E_\infty$ rings
    \[\begin{tikzcd}
        \THH(R)\arrow[r,"\varphi"]\arrow[d,"\pi"]&\THH(R)^{tC_p}\arrow[d,"\pi^{tC_p}"] \\
        R\arrow[r,dashed]&R^{tC_p} \\
    \end{tikzcd}\]
    where $R$ is equipped with the trivial $\T$-action. By the universal property of $\THH$, the non-equivariant $\E_\infty$ ring map underlying the dashed arrow must be the Frobenius; thus their definition asks for a $\T$-equivariant lift thereof. Since $R$ has the trivial action, this in turn is equivalent to a non-equivariant lift of the Frobenius into the $\T$-fixed points of $R^{tC_p}$, which by the Tate orbit lemma coincide with $(R^{t\T})_p$, matching our definition.
\end{rem}

Thus \Cref{lift_Frob_p_localization} implies that $\MU\l n\r$ is naturally a cyclotomic base, which at $n=2$ is an unpublished result of Hahn-Devalapurkar-Raksit-Yuan. 
\begin{rem}
    There is also the related work of Blumberg-Mandell-Yuan, who study ``$p$-precyclotomic base'' structures in a more refined setting--- namely the genuine equivariant setting. They show that certain genuine $\T$-equivariant lifts of $\MUP$ and $\MU$ admit cyclotomic base structures \cite[Theorem 7.4.3]{BMY}. In contrast, our result (and the definition of Hahn-Raksit-Wilson) is concerned with Borel-equivariant cyclotomic structures and with $\MU\l n\r$ equipped with the trivial $\T$-action. Thus, our result does not imply theirs. Their result does not imply ours either, because the genuine $\T$-structure on $\MU$ that they work with is an exotic one (cf. the discussion preceding \cite[Theorem 7.4.3]{BMY}).
\end{rem}
Using our formulas for $\#$ maps (\Cref{eulertate_formula}), we obtain the following cyclotomic rigidity result.

\begin{thm}\label{cyclo_rigidity}
    If $\phi:\MU\to \MU$ is a cyclotomic base map, then its underlying homotopy ring map is the identity.
\end{thm}
\begin{proof}
    Recall that $\xu$ denotes the universal coordinate on $\MU$, and that we also use it to denote the coordinate on $\MU^{t_\mathcal{F}G}$ induced by the unit map $\MU\to \MU^{t\mathcal{F}G}$.
    
    We must show that $\phi(\xu)=\xu$. As usual write $f(\xu)=\phi(\xu)$ for the power series expressing $\phi(\xu)$ in terms of $\xu$. Since $\phi$ is a cyclotomic self map of $\MU$, we have
    \[\psi_p\circ \phi(\xu)-\phi^{t\T}\circ\psi_p(\xu)=0.\]
    Now $\psi_p=\#_{\pi_p}$ by construction, and so by \Cref{eulertate_formula} (see also \ref{Frobcoordinate}) we have
    \[\psi_p(\xu)=\xu\cdot \tch_{\pi_p}(L-1) = \xu\cdot \prod_{k=1}^{p-1}\frac{\xu+_\Fu [k]_\Fu (\tu )}{[k]_\Fu (\tu )}\in\MU_*(\!(t)\!)_p[\![\xu ]\!].\]
    Using this, we expand the formula $\psi_p\circ \phi(\xu)-\phi^{t\T}\circ\psi_p(\xu)=0$ precisely as in \Cref{frob_commute_formula} but now the ambient ring is $\MU_*(\!(\tu)\!)_p[\![\xu ]\!]$ (as opposed to its quotient by $[p]_\Fu$) and find
    \[\psi_p(f)\left( \xu\cdot \prod_{k=1}^{p-1}\frac{ \xu+_{\Fu }[k]_{\Fu }(\tu)}{[k]_{\Fu }(\tu)} \right)-f(\xu)\cdot\prod_{k=1}^{p-1}\frac{f(\xu+_{\Fu }[k]_{\Fu }(\tu))}{f([k]_{\Fu }(\tu))}=0.\]
    We will take $p=2$ and explicitly extract the $\xu^2$ coefficient. Fix some notation: write 
    \begin{align*}
        \MU_*=&\ \Z[b_1,b_2,...] \\
        f(\xu)=&\ \xu+f_1\xu^2+f_2\xu^3+... \\
        g(\xu):=&\ \psi_2(\xu)= \xu+g_1(\tu)\xu^2 +g_2(\tu)\xu^3+... \\
    \end{align*}
    Thus e.g. $g_1(\tu)=\tu^{-1}+b_1+...$ is the coefficient of $\xu$ in $(\xu+_\Fu\tu)\tu^{-1}$.  Since $b_1$ is the coefficient of $\xu\yu$ in $\Fu(\xu,\yu)$, $\psi_2(b_1)$ is the coefficient of $\xu\yu$ in $(\psi_2)_*\Fu(\xu,\yu)=g(\Fu(g^{-1}(\xu),g^{-1}(\yu))$. This is $b_1+2g_1(\tu)$.  Next, observe that for degree reasons there is some $d\in\Z$ such that $f_1=db_1$. Putting this together, we find that the $\xu^2$ coefficient reads
    \[g_1(\tu)\left(2d+1-\frac{\sum_{j=0}(j+1)f_j\tu^j}{\sum_{j=0}f_j\tu^j}\right)=0.\]
    But from the $\tu$-expansion above and the fact that $\tu$ is topologically nilpotent we find that $g_1(\tu)$ is a unit, so we get
    \[2d+1=\frac{1+2f_1\tu+3f_2\tu^2+4f_3\tu^3...}{1+f_1\tu+f_2\tu^2+f_3\tu^3+...}.\]
    Comparing the $\tu^0$-coefficients shows that $d=0$, thus $f_1=0$. If $f_n$ is the first nonzero coefficient of $f(t)-t$, then examining the $\tu^{n}$-coefficient of the previous display we see that $nf_n=0$ so $f_n=0$. By induction we find that $f_n=0$ for all $n$ and we are done.

\end{proof}

\section{Obstruction theory for $\E_n$ orientations}\label{obstruction_section}
In this section we use the formula of \Cref{Frobcoordinate} to construct a general obstruction theory for $\E_n$ complex orientations of $\E_\infty$ rings. When applied to $p$-typical orientations, we use the calculations of Johnson-Noel \cite{JohnsonNoel} to get $\E_n$ obstructions for finite $n$. When applied to the Quillen idempotent specifically, we produce a computationally-compressed version of their obstruction and are able to achieve better bounds at small primes (cf. \Cref{compare_JN}). We then use this obstruction theory to prove the main results \Cref{thm_no_p2_E5} and \Cref{QI_p3}.

Recall that the logarithm of the universal formal group law $\Fu$ is the power series
\[\log_u(\xu)=\xu + m_1\frac{\xu^2}{2}+m_2\frac{\xu^3}{3}+m_3\frac{\xu^4}{4}+m_4\frac{\xu^5}{5}+...\]
where $m_i$ is the bordism class of $\CP^{i}$, a fact due to A. Mishchenko (cf. \cite[II Corollary 9.2]{Adams}). The logarithm of a formal group law $F_\gamma$ classified by a ring map $\gamma :\MU_*\to R_*$ is gotten by applying $\gamma$ to the coefficients of $\log_u$. The formal group law $F_\gamma$ is called \emph{$p$-typical} if $\gamma(m_i)=0$ when $i\neq p^j-1$, and if $R_*$ is $p$-local and $p$-torsion free.

\begin{notation}
Write $\rho$ for the regular representation of $C_p$. For $0\le n\le \infty$, write $\conf$ for the space of configurations of $p$ points in $\R^n$. For an $\EE_n$-ring $R$, we let 
\[
P_n^R\colon \Omega^\infty R \to \Omega^\infty R^{\conf_{hC_p}}
\]
be the $\EE_n$ total $p$-th power map. We use the same notation for the map in a general degree
\[
P_n^R\colon \Omega^\infty \Sigma^kR \to \Omega^\infty \left(\Sigma^{k\rho}R^{\conf}\right)^{hC_p}.
\]
\end{notation}

Thus, an $\EE_n$-map $\phi \colon R\to R'$ satisfies 
\[
\phi \circ P_n^R \simeq  P_n^{R'} \circ \phi.
\]
This gives the (very) classical first obstruction to an $\EE_n$-structure; for a general map $\phi \colon R\to R'$, if the difference $P_n^{R'}\circ \phi - \phi \circ P_n^R$ does not vanish on some class $a\in \pi_*R$, then the map $\phi$ cannot be upgraded to an $\EE_n$-map. Indeed, \cite{JohnsonNoel} carries out precisely this program for $n=\infty$ and $\phi$ a map inducing a $p$-typical formal group law. Motivated by their work, we make the following definition.

\begin{defn}\label{JN_obs_element}(cf. \cite[5.19]{JohnsonNoel})
    Let $\phi\colon \MU\to A$ be a homotopy ring map into an $\EE_\infty$-ring $A$. Write $[\CP^d]\in\MU_{2d}$ for the bordism class of projective space.  We define the \tdef{$d$-th JN-obstruction series} of $\phi$ to be the element 
\[
\JN_d := P_\infty^A(\phi([\CP^d])) - \phi(P_\infty^\MU([\CP^d])) \in A^{-2d\rho}(BC_p). 
\]
\end{defn}

Observe that we can evaluate $P_n^{R'}\circ \phi - \phi \circ P_n^R$ not only on elements of $R^*$, but on more general cohomology classes in $R^*X$ for a space $X$. 
When $R$ is complex orientable, there is a favorable choice of such class, namely the corresponding  coordinate in $R^{2}\CP^\infty$.  
This leads to the following definition.

\begin{defn}\label{def:obstruction_series}
Let $\phi\colon \MU\to A$ be a homotopy ring map into an $\EE_\infty$-ring $A$. Let $\xu$ be the universal coordinate in $\MU^*\CP^\infty$.  We define the \tdef{obstruction series} of $\phi$ to be the element 
\[
\obsser := P_\infty^A(\phi(\xu)) - \phi(P_\infty^\MU(\xu)) \in A^{2\rho}(BC_p\times \CP^\infty).  
\]
\end{defn}

\begin{rem}\label{rem_triv_rho}
Assume that $A$ comes equipped with a (possibly different) complex orientation $u:\MU\to A$ with coordinate $x$, and with respect to which the $p$-series $[p]_{\Fa}(x)$ is not a zero-divisor in $A^*[\![x]\!]$. Then via the Gysin sequence for the bundle $S^1\to BC_p\to \CP^\infty$ and the Thom isomorphism for the complex oriented bundle $2\rho$, we obtain an isomorphism 
\[
A^{2\rho}(BC_p\times \CP^\infty) \cong A^{2p}(BC_p\times \CP^\infty) \cong A_*[\![\ta,\xa]\!]/([p]_\Fa(t))[-2p],
\]
so that the class $\obsser$ can be thought of as a power series in the variables $\ta$ and $\xa$ (and the class $\JN_d$ a power series in $\ta$), well-defined modulo the $p$-series of the formal group law of $A$. This is our justification for the name ``obstruction series''.
\end{rem}
An immediate consequence of the above discussion is the following:
\begin{cor}\label{cor:obstruction_vanishes}
In the setting of \Cref{def:obstruction_series},
if $\phi\colon \MU \to A$ is an $\EE_n$-map, then the images of $\JN_d$ and $\obsser$ in $A^{2*\rho}(\conf_{hC_p}\times \CP^\infty)$ vanish.
\end{cor}
The first step in making this computationally useful is the following.

\begin{prop}\label{prop_obs_vanishing_trived}
    Suppose we are in the setting of \Cref{rem_triv_rho}, and moreover that $A$ is $p$-local, $p$-torsion free, and concentrated in even degrees. Identify $\JN_d$ and $\obsser$ with series 
    \[\JN_d(t)\in A_*[\![\ta]\!]/([p]_\Fa(t))[-2dp]\]
    \[\obsser(t)\in A_*[\![\ta,\xa]\!]/([p]_\Fa(t))[-2p].\]
    If $\phi$ is an $\E_n$ map for $n\geq 3$ odd, then these expressions vanish modulo $(t^{\frac{ (n-1)(p-1)}{2}+1})$.
\end{prop}
\begin{proof}
    We must investigate the restriction map 
    \[A^{2*\rho}(BC_p\times \CP^\infty)\to A^{2*\rho}(\conf_{hC_p}\times \CP^\infty).\] 
    Using the orientation $u$ to trivialize $\rho$ this becomes the degree $2p*$ part of the restriction map \[A^{*}(BC_p\times \CP^\infty)\to A^{*}(\conf_{hC_p}\times \CP^\infty).\]  
    To finish the proof, it suffices to show that this map detects the elements $t^r$ for  $r< \frac{(n-1)(p-1)}{2}+1$.

    By \cite[Theorem 6.1]{BLZ} the restriction map on mod $p$ cohomology
    \[\F_p^{*}(BC_p\times \CP^\infty)\to \F_p^{*}(\conf_{hC_p}\times \CP^\infty)\]
    detects the elements $t^r$ for  $r< \frac{(n-1)(p-1)}{2}+1$. Thus, the same map on integral cohomology detects the mod $p$ residue of these elements. We claim that the Atiyah-Hirzebruch spectral sequence for computing $A$-valued cohomology collapses in the relevant range of degrees, namely degrees at most $(n-1)(p-1)$, thereby finishing the proof. First, consider the integral Serre spectral sequence
    \[H^a(BC_p;H^b(\conf;\Z_{(p)}))\Rightarrow H^{a+b}(\conf_{hC_p};\Z_{(p)}). \]
    By \cite[Theorem 3.1]{BLZ}, the $E_2$-page is concentrated in degrees with $b$ divisible by $n-1$ (which is even), and for $0<b<(n-1)(p-1)$ the $C_p$-module structure is induced from the trivial subgroup, and therefore will only contribute group cohomology when $a=0$. Thus both $H^*(\conf; \mathbb{Z}_{(p)})$ and $H^*(\conf_{hC_p}; \mathbb{Z}_{(p)})$ are concentrated in even degrees in the desired range. Since $A_*$ is even concentrated and $p$-torsion free, the Atiyah-Hirzebruch spectral sequence for computing the $A$-valued cohomology collapses for degree reasons, in the desired range of degrees.
\end{proof}

Now that we have a computable vanishing condition on the series $\JN_d(t)$ and $\obsser(t)$, we need to compute the series themselves. When $\phi$ induces a $p$-typical formal group law (and $p\leq 13$), Johnson and Noel have done extensive calculations of $\JN_d$ \emph{modulo the reduced $p$-series} \[\l p\r_{\Fa}(t):=[p]_F(t)/t,\] which we will put to use. For general $\phi$ we will use $\obsser(t)$, and our goal will be achieved using our formula for $\Fr^{\MU}(\xu)$; under suitable assumptions, we will derive from it a formula for $\obsser(t)$, also modulo the reduced $p$-series.

\begin{prop}\label{prop_formula_for_obs_t_inverted}
In the setting of \Cref{prop_obs_vanishing_trived}, write $f(x)\in A^*[\![\xa]\!]$ for the coordinate defined by $\phi(\xu)=f(x)$ and identify $\obsser$ with a series $\obsser(t)$. Then after inverting $t$, we have
\[
\obsser(t) = \left(\prod_{k=1}^{p-1}[k]_F(t)\right)\left(\Fr^A(f)\left( \xa \prod_{k=1}^{p-1}\frac{ \xa+_{\Fa }[k]_{\Fa }(\ta)}{[k]_{\Fa }(\ta)} \right)-f(\xa)\prod_{k=1}^{p-1}\frac{f(\xa+_{\Fa }[k]_{\Fa }(\ta))}{f([k]_{\Fa }(\ta))}\right) \in A^*(\!(\ta)\!)/([p]_\Fa(t))[\![\xa]\!].
\]
\end{prop}

\begin{proof}
By \cite[IV.1.21]{NikolausScholze}, for any $\E_\infty$ ring $R$ we have a commutative diagram 
   \[ \begin{tikzcd}
       \Omega^\infty\Sigma^2R\arrow[r,"P_\infty^R"]\arrow[dr,"\Fr^R"]  & \Omega^\infty\left(\Sigma^{2\rho}R\right)^{hC_p}\arrow[d,"\mathrm{can}"]\\
   &\Omega^\infty\Sigma^2R^{tC_p}  \\
    \end{tikzcd}.\]
    The diagram is also natural along $\E_\infty$ ring maps $R\to R'$. We consider this diagram with $R=A$, and for $N>0$ we map $\CP^{N-1}$ in. As above, using the complex orientation $u$ of $A$ we make the identification
    \[A^{2\rho}(BC_p\times\CP^{N-1})\simeq A^{2p}(BC_p\times\CP^{N-1}).\] Then on $\pi_{0}$ the diagram becomes

     \[ \begin{tikzcd}
       A_*[\xa ]/(\xa ^{N})[-2]\arrow[r,"P_\infty^A"]\arrow[dr,"\Fr^A"]  & \frac{A_*[\![\ta ]\!]}{[p]_\Fu (\ta )}[\xa ]/(\xa ^{N})[-2p]\arrow[d,"\mathrm{can} "] \\
   &\frac{A_*(\!(\ta )\!)}{[p]_{\Fa }(\ta )}[\xa ]/(\xa ^{N})[-2]  \\
    \end{tikzcd}.\]
    Under these identifications the vertical map $\mathrm{can}$ is given by division by the Euler class of $\rho-1$, which is $\prod_{k=1}^{p-1}[k]_F(t)$. 
    
    For any ring map $\phi:\MU\to A$ consider the element $P_\infty^A(\phi(\xu))-\phi(P_\infty^\MU(\xu))$ in the middle ring and the element $\Fr^A(\phi(\xu))-\phi(\Fr^\MU(\xu))$ in 
    the bottom ring. By commutativity the vertical map sends the former to the latter.

     It remains to show that the right-hand parenthesized factor of the displayed element in the proposition statement represents 
     \[\Fr^A(\phi(\xu))-\phi(\Fr^\MU(\xu)),\] which follows from \Cref{frob_commute_formula}.

\end{proof}

By combining our formulas for the obstruction series and the computable vanishing condition of \Cref{prop_obs_vanishing_trived}, we can finally assemble our main computational tools. First we have the result for $\obsser(t)$.

\begin{prop}\label{prop_final_obsser_obs}
    In the setting of \Cref{prop_obs_vanishing_trived}, suppose that $\phi:\MU\to A$ is $\E_n$ for some $n=3,5,...,\infty$. Then after adding some multiple of $\l p\r_F(t)$, the expression
    \[ \obsser(t)=\left(\prod_{k=1}^{p-1}[k]_F(t)\right)\left(\Fr^A(f)\left( \xa \prod_{k=1}^{p-1}\frac{ \xa+_{\Fa }[k]_{\Fa }(\ta)}{[k]_{\Fa }(\ta)} \right)-f(\xa)\prod_{k=1}^{p-1}\frac{f(\xa+_{\Fa }[k]_{\Fa }(\ta))}{f([k]_{\Fa }(\ta))}\right) \in A^*(\!(\ta)\!)/([p]_\Fa(t))[\![\xa]\!]\]
    
    lifts to $A^*[\![\ta]\!]/([p]_\Fa(t))[\![\xa]\!]$  in such a way that this lift vanishes in $A^*[\![\ta]\!]/([p]_\Fa(t),t^{\lfloor\frac{ (n-1)(p-1)}{2}\rfloor+1})[\![\xa]\!]$.

\end{prop}
\begin{rem}
    Phrased more computationally: for the lift of each coefficient of $x$, one can iteratively add multiples of $\l p\r_F(t)$ to clear the term of lowest $t$-degree, until this term is $O(t^{\lfloor\frac{ (n-1)(p-1)}{2}\rfloor+1})$.
\end{rem}
\begin{proof}

    Consider the map \[A^*[\![\ta]\!]/([p]_\Fa(t))[\![\xa]\!]\to A^*(\!(\ta)\!)/([p]_\Fa(t))[\![\xa]\!].\]
    Since $A^*$ is $p$-torsion free, and $\l p\r_F(0)=p$, the kernel of this map is the ideal generated by $\l  p\r_F(t)$.  By \Cref{prop_formula_for_obs_t_inverted}, the displayed expression represents the image of $\obsser(t)$ under this map. Thus, after adding some multiple of $\l p\r_F(t)$, it lifts along it. Then since $\phi$ is $\E_n$ we can apply \Cref{prop_obs_vanishing_trived} to get the desired vanishing modulo $(t^{\lfloor\frac{ (n-1)(p-1)}{2}\rfloor+1})$. 
\end{proof}
As Johnson-Noel have already made extensive calculations of $\JN_d(t)$, to use their formulas we must simply translate them into our notation.
\begin{lem}\label{translate_MCns}
    Let $p\leq 13$. Suppose that $\phi$ is the Quillen idempotent $\MU\to\MU_{(p)}$ and write $v_i\in(\MU_{(p)})_{2(p^i-1)}$ for the Hazewinkel generators. Consider the elements which Johnson-Noel denote by $MC_n(\xi)$  (cf. \cite{JohnsonNoel} 5.19 and 6.3-6.8). Substitute $\xi\to \tu$ and choose any $d\neq p^j-1$. Then up to multiples of $\l p \r_F(\tu)$ we have
    \[\JN_d(\tu)=-MC_d(\tu)\Big(((p-1)!)^{-2d}\tu^{2d(1-p)}+O(\tu^{2d(1-p)+1})\Big)\in \MU_*(\!(\tu)\!)/([p]_F(\tu)).\]
\end{lem}
\begin{proof}
    Since $d\neq p^j-1$ and $\phi$ (which Johnson-Noel call $r$) is $p$-typical we have $\phi([\CP^d])=0$. Thus
    \[\JN_d(\tu)=P_\infty^{\MU_{(p)}}(\phi([\CP^d])) - \phi(P_\infty^\MU([\CP^d]))= - \phi(P_\infty^\MU([\CP^d])).\]
    The element in \cite{JohnsonNoel} denoted by $\chi$ is the euler class of $\rho-1$, which is 
    \[\chi:=\prod_{k=1}^{p-1}[k]_F(\tu)=(p-1)!\tu^{p-1}+O(\tu^p).\]
    For the rest of the proof we work modulo $\l p \r_F(\tu)$. Then by definition (\cite[5.19]{JohnsonNoel})  we have $MC_d(\tu)=\phi(\chi^{2d}P_\infty^\MU([\CP^d]))$ and  we get
    \begin{align*}
        \JN_d(\tu)&=- \phi(P_\infty^\MU([\CP^d]))\\
        &=- \phi(\chi^{2d}P_\infty^\MU([\CP^d])\chi^{-2d})\\
        &=- \phi(\chi^{2d}P_\infty^\MU([\CP^d]))\cdot\phi(\chi)^{-2d}\\
        &=- MC_d(\tu)\phi(\chi)^{-2d}
    \end{align*}
    
    and so the result follows from the displayed $\tu$-expansion of $\chi$.
\end{proof}

\begin{prop}\label{final_jn_obs}
    In the setting of \Cref{prop_final_obsser_obs}, suppose that $\phi:\MU\to A$ is $\E_n$ and induces a $p$-typical formal group law. Then, after adding some multiple of $\l p \r_F(t)$, the element
    \[\JN_d(t)=-\phi(MC_d)(t)\Big(((p-1)!)^{-2d}t^{2d(1-p)}+O(t^{2d(1-p)+1})\Big)\in A_*(\!(t)\!)/([p]_F(t)).\]
    lifts to $A^*[\![\ta]\!]/([p]_\Fa(t))$ in such a way that it vanishes in $A^*[\![\ta]\!]/([p]_\Fa(t),t^{\lfloor\frac{ (n-1)(p-1)}{2}\rfloor+1})$.
\end{prop}
\begin{proof}
    Since $\phi$ is $p$-typical, we may take the Johnson-Noel calculations of $\JN_d$ (\Cref{translate_MCns}) and push them forward along $\phi$. Note that $\phi(MC_d(\tu))=\phi(MC_d)(t)$ (cf. \Cref{pushforward_series}). This establishes the displayed equality. The proof of the vanishing condition is identical to that in \Cref{prop_final_obsser_obs}.
\end{proof}

We can now state two applications of this obstruction theory.

\begin{thm}\label{thm_no_p2_E5}
    Suppose $R$ is an $\E_\infty$ ring with nonzero $T(1)$-localization. Then there does not exist a complex orientation of $R$ which is both $\E_5$ and $2$-typical.
\end{thm}
\begin{proof} 
 By the chromatic Nullstellensatz of \cite{nullstel}, $R$ admits an $\E_\infty$ map to a height 1 Morava $E$-theory $E_1(\kappa)$ of an algebraically closed field $\kappa$ of characteristic $2$. Thus it suffices to assume $R=E_1(\kappa)$. 


 Now suppose we have a complex orientation $\phi:\MU\to E_1(\kappa)$ which is $2$-typical. Then using \Cref{translate_MCns} and the formulas for $MC_2(\xi)$ and $MC_4(\xi)$ in \cite[6.3]{JohnsonNoel}, we find that, up to a multiple\footnote{Of course, Johnson-Noel have already performed the long division with $\l 2\r_F(t)$, so this is somewhat redundant. } of $\l 2\r_F(t)$ we have
 \[\JN_2(t)=\left(t^6(\phi(v_1)^6+\phi(v_2)^2)+O(t^7)\right)
 (t^{-4})=t^{2}(\phi(v_1)^6+\phi(v_2)^2)+O(t^3)\]
 \[\JN_4(t)=\left(t^{10}\phi(v_1)^4\phi(v_2)^2+O(t^{11})\right)
 (t^{-8})=t^{2}\phi(v_1)^4\phi(v_2)^2+O(t^3).\]
 By \Cref{final_jn_obs}, the $t^2$ terms must be clearable by addition of a multiple of $\l 2\r_\F(t)$. But then modulo 2,  $\phi(v_1)^6=\phi(v_2)^2$ from $\JN_2(t)$, so $\phi(v_1)^{10}=0$ from $\JN_4(t)$. But $\phi$ induces a height 1 formal group law over $\pi_*E_1(\kappa)$, which by definition means that $\phi(v_1)$ is invertible in $\pi_*E_1(\kappa)/2$.

\end{proof}

\begin{rem}\label{rem_JN_diff1}
    We obtain analogous results for $p=3,5,7,11,13$. For $p$ in this range, we use the formula of \cite{JohnsonNoel} for $MC_{2(p-1)}(\xi)$ and the exact same proof as above, to find that for an $\E_\infty$ ring $R$ with $L_{T(1)}R\neq 0$ no orientation of $R$ is simultaneously $p$-typical and $\E_{2p+3}$.
\end{rem}

\begin{rem}
    We obtain as an immediate corollary that the Quillen idempotent at $p=2$ is not $\E_5$ (and not $E_{2p+3}$ at $p=3,...,13$). The best previously known bound is due to Lawson \cite[Remark 4.4.7]{LawsonBP}, who showed it is not $\E_{7}$ (at $p=3,...,13$ we match Senger's $\E_{2p+3}$ obstruction \cite[Theorem 1.3]{Senger}).
\end{rem}

Now we showcase that the obstruction theory using $\obsser(t)$ can do a bit better.

\begin{thm}\label{QI_p3}
    The Quillen idempotent at $p=3$ is not $\E_5$.
\end{thm}
\begin{proof}
     Consider the (3-localized) Todd orientation $\Td:\MU_{}\to\ku_{(3)}$. This is an $\E_\infty$ orientation (cf. e.g. \cite{Joachim}, \cite[Theorem 10.3]{ahr}). Thus, it suffices to show that the composite $\MU\to\MU_{(3)}\oto{\Td}\ku_{(3)}$ classifying the $3$-typification of the Todd orientation is not $\E_5$. Call this composite $\phi$.
     
    We have $(\ku_{(3)})_*\simeq \Z_{(3)}[\beta]$ with $|\beta|=2$, and the formal group law associated to the Todd orientation is given by 
    \[\xa+_\Td \ya=\xa+\ya-\beta \xa\ya,\]
    the 3-series is given by
    \[[3]_\Td(\ta)=3\ta-3\beta \ta^2+\beta^2 \ta^3,\]
    and the logarithm is
    \[\mathrm{log}_\Td(\xa)=\xa + \beta\frac{\xa^2}{2}+\beta^2\frac{\xa^3}{3}+\beta^3\frac{\xa^4}{4}+\beta^4\frac{\xa^5}{5}+...\]
    The logarithm of the 3-typification is 
    \[\log_\phi(x)=x+\beta^2 \frac{x^3}{3}+\beta^8 \frac{x^9}{9} + ...\]
    Written in terms of the Todd coordinate, $f(x)=\phi(\xu)$ becomes
    \[f(x):=\exp_\phi\circ \log_\Td (x) = x + \frac{\beta}{2} x^2 - \frac{\beta^3}{4} x^4 - 
 \frac{\beta^4}{20} x^5+....\]
    Finally, we have $\Fr(\beta^k)=3^k\beta^k$ by \cite[IV.1.12]{NikolausScholze}. Using these facts, we can explicitly expand the expression in \Cref{prop_final_obsser_obs}. If $\phi$ is an $\E_5$ map, then upon long division of that expression by $\l 3\r _\Td(t)$, the remainder must be $O(t^5)$. With computer assistance (\Cref{appendix}) we apply this to the coefficient of $x^9$. After expanding, the coefficient of $x^9$ is 
    \[ - \frac{1215 \beta^3}{8 t^3} + \frac{15957 \beta^4}{32 t^2} + 
  \frac{1041093 \beta^5}{64 t} -\frac{15676713 \beta^6}{448}+...\]
  \[...+\frac{8826507 \beta^7 t}{2240} + \frac{304379169 \beta^8 t^2}{
  17920} - \frac{334416123 \beta^9 t^3}{35840} - \frac{37570767 \beta^{10} t^4}{6400} \]
    and after using $3-3\beta \ta+\beta^2 \ta^2$ to successively clear the term of lowest $t$-degree the remainder (modulo 3) is $\beta^{10}t^4$.
\end{proof}

\begin{rem}
Recall that $\KU^{tC_p}\cong \QQ_p(\zeta_p)[\beta^{\pm 1}]$, under which the class $t$ maps to $-\beta^{-1}(\zeta_p -1)$. Accordingly,  the above computation can be interpreted as a computation in the power series ring $\QQ_p(\zeta_p)[\![x]\!]$, where the required divisibility condition translates to the $p$-adic valuation of the resulting expression being $\ge 2$. 
\end{rem}


\begin{rem}\label{compare_JN}
    The obstruction theory with $\obsser$ has the following benefits over the one using $\JN$ and the Johnson-Noel calculations: first, it does not require $\phi$ to be $p$-typical. Second, because of the formula
\[P_\infty^\MU(\xu)=\prod_{k=1}^{p-1}[k]_\Fu(\tu)\cdot \Fr(\xu)=\xu\prod_{k=1}^{p-1}\xu+_\Fu [k]_\Fu(\tu),\]
which pushes forward nicely under arbitrary complex orientations, we are able to compose the Quillen idempotent with a map $\MU_{(p)}\to \ku_{(p)}$ and thus make all our computations in the ground ring $(\ku_{(p)})_*\simeq \Z_{(p)}[\beta]$. In contrast, there is no such formula for $P_\infty^\MU[\CP^d]$, and so Johnson-Noel are left with a problem of much higher computational complexity since the ground ring is $(\MU_{(p)})_*\simeq \Z_{(p)}[b_1, b_2,...]$. 
We expect that one can find obstructions in a larger range of primes because of this. On the other hand, in the range where Johnson-Noel do find obstructions, their formulas are much stronger than a mere obstruction to the Quillen idempotent, as they occur in a universal ring.
\end{rem}

We end this section by outlining a streamlined procedure to calculate Johnson-Noel's elements $P_\infty^\MU([\CP^d])$ using the Frobenius formula of \Cref{Frobcoordinate}.

 Consider the logarithm of the formal group law corresponding to the orientation $\Fr:\MU \to \MU^{tC_p}$. On the one hand, as mentioned at the beginning of this section, the logarithm pushes forward under ring maps:
\[\log_\Fr(\xu)=\Fr_*(\log_u(\xu)):=\xu + \Fr(m_1)\frac{\xu^2}{2}+\Fr(m_2)\frac{\xu^3}{3}+\Fr(m_3)\frac{\xu^4}{4}+
\Fr(m_4)\frac{\xu^5}{5}+...\]
On the other hand, we know that the power series 
\[f_\Fr(\xu)=\xu\prod_{k=1}^{p-1}\frac{\xu+_\Fu [k]_\Fu (\tu )}{[k]_\Fu (\tu )}\]
is the coordinate transformation relating the Frobenius orientation and the unit orientation, so
\[\log_\Fr(\xu)=\log_u\circ f_\Fr^{-1}(\xu).\]
Since $m_i=[\CP^i]$, the elements  $P_\infty^\MU([\CP^d])=\left(\prod_{k=1}^{p-1}[k]_\Fu(\tu)\right)^{-d}\cdot\Fr([\CP^d])$ can be extracted from the coefficients of this composition of power series.

\section{Miscellaneous  applications}\label{section_further_app}
In this section we present some miscellaneous applications of \Cref{sharp_construction}.

\subsection{Height restrictions on $\MO\l n \r$-oriented rings}
Using the real sharp construction (cf. especially \Cref{remark_generalized_sharp}) we get for each $n$ a ring map
\[\MO\l n\r\to\MO\l n+1\r^{tC_2},\]
thereby obtaining a new proof of a result of Hovey-Ravenel  about the chromatic support of $\MO\l n\r$.
\begin{thm}\label{thm_height_MOn} (\cite[Corollary 5.4, (2)]{ravhovey})
    Let $R$ be an $\MO\l n\r$-oriented $\E_\infty$ ring. Let $\phi(n)$ be the number of nonzero homotopy groups of $\BO$ in the range $[1,n-1]$. Then, at $p=2$, $R$ is $T(\phi(n))$-acyclic.
\end{thm}
\begin{proof}
    It suffices to consider the (initial) case that $R=\MO\l n \r$. By \Cref{sharp_construction} we know that $R_1:=R^{tC_2}$ has an $\MO\l n-1\r$-orientation. Let $R_i:=R_{i-1}^{tC_2}$. Iterating this, we find that $R_{n-1}$ has an $\MO$-orientation and is thus an $\F_2$-algebra\footnote{Note that $\MO$ is not an $\E_\infty$-$\F_2$ algebra, cf. \cite{Szy14}.}. In particular it is $T(0)$-acyclic. By ``purity'' of Tate blueshift (\cite[Corollary 4.6.1]{hahnblue}), that means $R_{n-2}$ is $T(1)$-acyclic. Iterating again, we find $R$ is $T(n-1)$-acyclic. To get $T(\phi(n))$-acyclicity, simply note that $\MO\l k\r\to \MO\l k-1\r$ is an isomorphism if $\pi_{k-1}\BO=0$, so that already $R_{\phi(n)}$ is $\MO$-orientable.

\end{proof}

\begin{rem}
    On the topic of $\MO\l n \r$ orientations, we record the following consequence of the chromatic Nullstellensatz (\cite{nullstel}).
    Let $E_k$ be a Morava $E$-theory of height $k$ over an algebraically closed field of characteristic 2. Let $n_k$ be the smallest number such that $\MO\l n_k\r$ is not $T(k)$-acyclic, if this exists. Then $E_k$ admits an $\E_\infty$ orientation by $\MO\l n_k\r$.
\end{rem}

\subsection{Steenrod squares for $\MU$}
Using the sharp construction we can recover a construction of certain lifts of the Steenrod powers due to tom Dieck (\cite[21.1]{Dieck}).

Consider the lifts of Frobenius of \Cref{thm_lift_of_Frob}
\[\#_{\pi_p}:\MU\to\MU^{t\T}_p.\]
Identifying $(\MU^{t\T})_p^*=\MU_*(\!(\tu )\!)_p$, the induced map of cohomology theories, at a finite CW complex $X$, defines a sequence of $\MU_p$ cohomology operations $P^{-i}_\MU$ via the formula


\[
\#_{\pi_p}(\alpha) = \sum_iP_\MU^{-i}(\alpha)t^i.\]

For $i\geq 0$ write
\[P^i:\F_p\to \Sigma^{2(p-1)i}\F_p\] 
for the $i$th Steenrod power. At $p=2$ this coincides with $\Sq^{2i}$.

\begin{prop}
  For each prime $p$ and $i\geq 0$ there is a commutative square
\[\begin{tikzcd}
    \MU_p\arrow[r,"(-1)^iP^{(p-1)i}_\MU"]\arrow[d]&\Sigma^{2(p-1)i}\MU_p\arrow[d] \\
    \F_p\arrow[r,"P^i "]&\Sigma^{2(p-1)i}\F_p \\
\end{tikzcd}.\]  
\end{prop}
\begin{proof}
    Consider the diagram
    \[
\begin{tikzcd}
\MU\arrow[r,"\#_{\pi_p}"]\arrow[d]&\MU^{t\T}\arrow[d]\\
\Z\arrow[r,"\Fr_\Z"]&(\Z^{tC_p})\\
\end{tikzcd}\]
where the right vertical map is the composite of $\MU^{t\T}\to \MU^{tC_p}$ and naturality of the Tate construction $\MU^{tC_p}\to\Z^{tC_p}$. 
By \Cref{thm_lift_of_Frob} and naturality of the Frobenius \cite[IV.1.4]{NikolausScholze} the diagram commutes. Finally, the bottom row is identified with the required conglomerate of $P^i$ by \cite[IV.1.15]{NikolausScholze}. 
\end{proof}
Recall that there do not exist lifts of $\Sq^{2i}$ to $\Z$, because of the Adem relation $\Sq^1\Sq^{2i}=\Sq^{2i+1}\neq0$. 

With $\xu$ and $\tu$ as before, we have
\[\xu \frac{\xu+_\Fu \tu}{\tu}=\sum_i\Sq_\MU^{-i}(\xu)\tu^i,\]
so that unlike in the case of the usual Steenrod algebra, there are nontrivial negative Steenrod operations: for $i>1$
\[\Sq^{-i+1}_\MU(\xu)=\sum_ja_{ij}x^{j+1}.\]

\begin{ques}
It would be interesting to investigate what is the analog of Adem relations for the operations $\Sq^i_\MU$. In particular, one can speculate that a version of  Bullett-MacDonald identity of the form
\[
\sum_{ij}\Sq^i_\MU\Sq^j_\MU(\alpha)\tu^{-i}(\mathbf{s}\frac{\mathbf{s}+_\Fu \tu}{\tu})^{-j}=\sum_{ij}\Sq^i_\MU\Sq^j_\MU(\alpha)\mathbf{s}^{-i}(\tu \frac{\tu+_\Fu \mathbf{s}}{\mathbf{s}})^{-j}
\]
holds.

\end{ques}

\section{Appendix}\label{appendix}
What follows is a simple Mathematica script to expand the expression of \Cref{prop_final_obsser_obs} for the $p$-typification of the Todd orientation (used in \Cref{QI_p3}). The reader is encouraged to reach out to the authors for a .nb or .txt version if they so desire.

\noindent\(\pmb{\text{{``}Choose a prime.{''}};}\\
\pmb{p=3;}\\
\pmb{\text{{``}Choose an n for which you'd like to find E$\_$n obstructions{''}};}\\
\pmb{n=5;}\\
\pmb{\text{tbd}=(n-1)(p-1)/2;}\\
\pmb{\text{$\texttt{"}$Choose a cutoff on the x-degree to expand.
$\texttt{"}$};}\\
\pmb{\text{powchoice}=3;}\\
\pmb{\text{bd}=p{}^{\wedge}\text{powchoice};}\\
\pmb{}\\
\pmb{\text{{``}The FGL of the Todd orientation, its logarithm, and p-typical logarithm / exponential{''}};}\\
\pmb{\text{Fm}[\text{x$\_$},\text{y$\_$}]=x+y-b*x*y;}\\
\pmb{\text{logm}[\text{x$\_$}]=x;}\\
\pmb{\text{For}[i=1,i<\text{bd},i\text{++},\text{logm}[\text{x$\_$}]=\text{logm}[x]+b{}^{\wedge}(i)*x{}^{\wedge}(i+1)/(i+1)];}\\
\pmb{\text{logmptyp}[\text{x$\_$}]=x;}\\
\pmb{\text{For}[i=1,i<\text{powchoice}+1,i\text{++},\text{logmptyp}[\text{x$\_$}]=\text{logmptyp}[x]+b{}^{\wedge}(p{}^{\wedge}(i)-1)*x{}^{\wedge}(p{}^{\wedge}(i))/(p{}^{\wedge}(i))]}\\
\pmb{\text{expptypm}[\text{x$\_$}]=\text{Normal}[\text{InverseSeries}[\text{Series}[\text{logmptyp}[x],\{x,0,\text{bd}\}],x]];}\\
\pmb{}\\
\pmb{\text{{``}The quillen idempotent in terms of the Todd coordinate, and its Frobenius image{''}};}\\
\pmb{\text{QIp}[\text{x$\_$}]=\text{Normal}[\text{Series}[\text{expptypm}[\text{logm}[x]],\{x,0,\text{bd}\}]];}\\
\pmb{\text{FrQIp}[\text{x$\_$}]=\text{QIp}[x]\text{/.}\{b\text{-$>$}p*b\};}\\
\pmb{}\\
\pmb{\text{{``}The k-series and the euler class of the reduced regular representation{''}};}\\
\pmb{\text{Ser}[1,\text{t$\_$}]=t;}\\
\pmb{\text{For}[i=2,i<p+1,i\text{++}, \text{Ser}[i,\text{t$\_$}]=\text{Fm}[t,\text{Ser}[i-1,t]]];}\\
\pmb{\text{Eulp}[\text{t$\_$}]=t;}\\
\pmb{\text{For}[i=2,i<p,i\text{++},\text{Eulp}[\text{t$\_$}]=\text{Normal}[\text{Series}[\text{Eulp}[t]*\text{Ser}[i,t],\{t,0,\text{bd}\}]]];}\\
\pmb{}\\
\pmb{\text{{``}The Frobenius coordinate and its Frobenius image.{''}};}\\
\pmb{\text{Frobcoordp}[\text{x$\_$},\text{t$\_$}]=x;}\\
\pmb{\text{For}[i=1,i<p,i\text{++},\text{Frobcoordp}[\text{x$\_$},\text{t$\_$}]=\text{Normal}[\text{Series}[\text{Frobcoordp}[x,t]*\text{Fm}[x,\text{Ser}[i,t]}\\
\pmb{]*\text{Ser}[i,t]{}^{\wedge}(-1),\{x,0,\text{bd}\}]]];}\\
\pmb{}\\
\pmb{\text{{``}The numerator and denominator of the right-hand term of the obstruction series{''}};}\\
\pmb{\text{QIpFrob}[\text{x$\_$},\text{t$\_$}]=\text{QIp}[x];}\\
\pmb{\text{For}[i=1,i<p,i\text{++},\text{QIpFrob}[\text{x$\_$},\text{t$\_$}]=\text{Normal}[\text{Series}[\text{QIpFrob}[x,t]*\text{QIp}[\text{Fm}[x,\text{Ser}[i,t]]]}\\
\pmb{,\{x,0,\text{bd}\}]]];}\\
\pmb{\text{QIEulp}[\text{t$\_$}]=\text{QIp}[t];}\\
\pmb{\text{For}[i=2,i<p,i\text{++},\text{QIEulp}[\text{t$\_$}]=\text{Normal}[\text{Series}[\text{QIEulp}[t]*\text{QIp}[\text{Ser}[i,t]],\{t,0,\text{bd}\}]]];}\\
\pmb{}\\
\pmb{\text{{``}The obstruciton series{''}};}\\
\pmb{\text{FrCommKUp}[\text{x$\_$},\text{t$\_$}]=\text{Normal}[\text{Series}[\text{Normal}[\text{Series}[\text{Eulp}[t](\text{FrQIp}[\text{Frobcoordp}[x,t]]-}\\
\pmb{\text{QIpFrob}[x,t]*\text{QIEulp}[t]{}^{\wedge}(-1)),\{x,0,\text{bd}\}]],\{t,0,\text{tbd}\}]];}\\
\pmb{}\\
\pmb{\text{$\texttt{"}$The obstruction series after adding multiples of the (reduced) p-series to clear terms of }}\\
\pmb{\text{t-degree leq tbd (long division) $\texttt{"}$};}\\
\pmb{\text{FrCommKUpclean}[\text{x$\_$},\text{t$\_$}]=\text{FrCommKUp}[x,t];}\\
\pmb{\text{For}[i=1,i<\text{bd},i\text{++},\text{FrCommKUpclean}[\text{x$\_$},\text{t$\_$}]=\text{Normal}[\text{Series}[\text{FrCommKUpclean}[x,t]-}\\
\pmb{\text{Coefficient}[\text{FrCommKUpclean}[x,t],t{}^{\wedge}(i-\text{bd})]/p*t{}^{\wedge}(i-\text{bd}-1)*\text{Ser}[p,t],\{t,0,\text{tbd}\}]]];}\\
\pmb{\text{FrCommKUpclean}[\text{x$\_$},\text{t$\_$}]=\text{Normal}[\text{Series}[\text{FrCommKUpclean}[x,t]}\\
\pmb{-\text{Coefficient}[\text{FrCommKUpclean}[x,t],t{}^{\wedge}(-1)]/p*t{}^{\wedge}(-2)*\text{Ser}[p,t],\{t,0,\text{tbd}\}]];}\\
\pmb{\text{FrCommKUpclean}[\text{x$\_$},\text{t$\_$}]=\text{Normal}[\text{Series}[\text{FrCommKUpclean}[x,t]}\\
\pmb{-\text{Coefficient}[\text{FrCommKUpclean}[x,t]*t,t{}^{\wedge}(1)]/p*t{}^{\wedge}(-1)*\text{Ser}[p,t],\{t,0,\text{tbd}\}]];}\\
\pmb{\text{For}[i=1,i<\text{tbd},i\text{++},\text{FrCommKUpclean}[\text{x$\_$},\text{t$\_$}]=\text{Normal}[\text{Series}[\text{FrCommKUpclean}[x,t]-}\\
\pmb{\text{Coefficient}[\text{FrCommKUpclean}[x,t],t{}^{\wedge}(i)]/p*t{}^{\wedge}(i-1)*\text{Ser}[p,t],\{t,0,\text{tbd}\}]]];}\\
\pmb{}\\
\pmb{\text{$\texttt{"}$The final obstruction
$\texttt{"}$};}\\
\pmb{\text{PolynomialMod}[\text{CoefficientList}[\text{FrCommKUpclean}[x,t],\{x\},\{\text{bd}\}],p]\text{//}\text{TableForm}}\)

\bibliographystyle{halpha}

\bibliography{references}

\newcommand{\etalchar}[1]{$^{#1}$}
\begin{thebibliography}{ABG{\etalchar{+}}14}

\bibitem[ABG{\etalchar{+}}14]{abghr}
Matthew Ando, Andrew~J. Blumberg, David Gepner, Michael~J. Hopkins, and Charles Rezk.
\newblock Units of ring spectra, orientations and {T}hom spectra via rigid infinite loop space theory.
\newblock {\em J. Topol.}, 7(4):1077--1117, 2014.

\bibitem[ACB19]{barthelthom}
Omar Antol\'in-Camarena and Tobias Barthel.
\newblock A simple universal property of {T}hom ring spectra.
\newblock {\em J. Topol.}, 12(1):56--78, 2019.

\bibitem[Ada74]{Adams}
J.~F. Adams.
\newblock {\em Stable homotopy and generalised homology}.
\newblock Chicago Lectures in Mathematics. University of Chicago Press, Chicago, Ill.-London, 1974.

\bibitem[AFG08]{afg}
Matthew Ando, Christopher~P. French, and Nora Ganter.
\newblock The {J}acobi orientation and the two-variable elliptic genus.
\newblock {\em Algebr. Geom. Topol.}, 8(1):493--539, 2008.

\bibitem[AHR10]{ahr}
M.~Ando, M.~Hopkins, and C.~Rezk.
\newblock Multiplicative orientations of $ko$-theory and the spectrum of topological modular forms.
\newblock 2010.

\bibitem[Bal23]{Balderrama}
William Balderrama.
\newblock Algebraic theories of power operations.
\newblock {\em J. Topol.}, 16(4):1543--1640, 2023.

\bibitem[BLZ15]{BLZ}
Pavle V.~M. Blagojevi\'c, Wolfgang L\"uck, and G\"unter~M. Ziegler.
\newblock Equivariant topology of configuration spaces.
\newblock {\em J. Topol.}, 8(2):414--456, 2015.

\bibitem[BMY26]{BMY}
Andrew~J. Blumberg, Michael~A. Mandell, and Allen Yuan.
\newblock Relative cyclotomic structures and equivariant complex cobordism.
\newblock {\em Proc. Lond. Math. Soc. (3)}, 133(4):Paper No. e70207, 2026.

\bibitem[BSY24]{nullstel}
R.~Burklund, T.~Schlank, and A.~Yuan.
\newblock The chromatic nullstellensatz.
\newblock {\em Ann. of Math}, 2024.

\bibitem[Car23]{CarmeliPic}
Shachar Carmeli.
\newblock On the strict {P}icard spectrum of commutative ring spectra.
\newblock {\em Compos. Math.}, 159(9):1872--1897, 2023.

\bibitem[CL25]{tatevaluedchar}
Shachar Carmeli and Kiran Luecke.
\newblock Tate-valued characteristic classes, 2025, 2503.12134.

\bibitem[Dev25]{sanath}
Sanath Devalapurkar.
\newblock Spherochromatism in representation theory and arithmetic geometry, 2025.
\newblock Harvard University PhD Thesis.

\bibitem[GM95a]{GreenleesMaybook}
J.~P.~C. Greenlees and J.~P. May.
\newblock Equivariant stable homotopy theory.
\newblock In {\em Handbook of algebraic topology}, pages 277--323. North-Holland, Amsterdam, 1995.

\bibitem[GM95b]{GreenleesMay}
J.~P.~C. Greenlees and J.~P. May.
\newblock Generalized {T}ate cohomology.
\newblock {\em Mem. Amer. Math. Soc.}, 113(543):viii+178, 1995.

\bibitem[Hah16]{hahnblue}
J.~Hahn.
\newblock On the \text{Bousfield} classes of $\text{H}_\infty$-ring spectra.
\newblock 2016.

\bibitem[HL18]{hoplaw}
Michael~J. Hopkins and Tyler Lawson.
\newblock Strictly commutative complex orientation theory.
\newblock {\em Math. Z.}, 290(1-2):83--101, 2018.

\bibitem[HR95]{ravhovey}
Mark~A. Hovey and Douglas~C. Ravenel.
\newblock The {$7$}-connected cobordism ring at {$p=3$}.
\newblock {\em Trans. Amer. Math. Soc.}, 347(9):3473--3502, 1995.

\bibitem[HRW22]{HRW}
Jeremy Hahn, Arpon Raksit, and Dylan Wilson.
\newblock A motivic filtration on the topological cyclic homology of commutative ring spectra.
\newblock 2022.

\bibitem[HY20]{hahnyuan}
Jeremy Hahn and Allen Yuan.
\newblock Exotic multiplications on periodic complex bordism.
\newblock {\em J. Topol.}, 13(4):1839--1852, 2020.

\bibitem[JN10]{JohnsonNoel}
Niles Johnson and Justin Noel.
\newblock For complex orientations preserving power operations, {$p$}-typicality is atypical.
\newblock {\em Topology Appl.}, 157(14):2271--2288, 2010.

\bibitem[Joa04]{Joachim}
Michael Joachim.
\newblock Higher coherences for equivariant {$K$}-theory.
\newblock In {\em Structured ring spectra}, volume 315 of {\em London Math. Soc. Lecture Note Ser.}, pages 87--114. Cambridge Univ. Press, Cambridge, 2004.

\bibitem[Law18]{LawsonBP}
Tyler Lawson.
\newblock Secondary power operations and the {B}rown-{P}eterson spectrum at the prime 2.
\newblock {\em Ann. of Math. (2)}, 188(2):513--576, 2018.

\bibitem[Lew78]{Lewis}
Lemoine Gaunce~Jr Lewis.
\newblock {\em T{HE} {STABLE} {CATEGORY} {AND} {GENERALIZED} {THOM} {SPECTRA}}.
\newblock ProQuest LLC, Ann Arbor, MI, 1978.
\newblock Thesis (Ph.D.)--The University of Chicago.

\bibitem[Nov67]{Novikov}
S.~P. Novikov.
\newblock Rings of operations and spectral sequences of {A}dams type in extraordinary cohomology theories. {$U$}-cobordism and {$K$}-theory.
\newblock {\em Dokl. Akad. Nauk SSSR}, 172:33--36, 1967.

\bibitem[NS18]{NikolausScholze}
Thomas Nikolaus and Peter Scholze.
\newblock On topological cyclic homology.
\newblock {\em Acta Math.}, 221(2):203--409, 2018.

\bibitem[Qui71]{Quillen}
Daniel Quillen.
\newblock Elementary proofs of some results of cobordism theory using {S}teenrod operations.
\newblock {\em Advances in Math.}, 7:29--56, 1971.

\bibitem[Sen24]{Senger}
Andrew Senger.
\newblock The {B}rown-{P}eterson spectrum is not {$\Bbb{E}_{2 ( p^2 + 2 )}$} at odd primes.
\newblock {\em Adv. Math.}, 458:Paper No. 109996, 33, 2024.

\bibitem[Szy14]{Szy14}
Markus Szymik.
\newblock Commutative {$\mathbb{S}$}-algebras of prime characteristics and applications to unoriented bordism.
\newblock {\em Algebraic \& Geometric Topology}, 14(6):3717--3743, 2014.

\bibitem[tD68]{Dieck}
Tammo tom Dieck.
\newblock Steenrod-{O}perationen in {K}obordismen-{T}heorien.
\newblock {\em Math. Z.}, 107:380--401, 1968.

\end{thebibliography}

\end{document}